\documentclass{amsart}
\usepackage{graphicx,amsfonts,amsmath,amssymb,amsthm,mathrsfs,url,hyperref} 
\usepackage[style=alphabetic]{biblatex} 
\usepackage[margin=1in]{geometry}
\usepackage[colorinlistoftodos]{todonotes}
\bibliography{refs}
\usepackage[all]{xy}

\theoremstyle{plain}
\newtheorem{thm}{Theorem}
\newtheorem{lemma}[thm]{Lemma}

\theoremstyle{definition}
\newtheorem{defn}[thm]{Definition}

\theoremstyle{remark}

\newtheorem*{ex}{Example}

\newtheorem*{notn}{Notation}

\numberwithin{equation}{section}
\numberwithin{thm}{section}

\DeclareMathOperator{\cl}{cl}

\DeclareMathOperator{\mex}{mex}

\title{Finite Cardinalities of Mis\`ere quotients}
\author{Simon Rubinstein-Salzedo}
\address{Euler Circle, Mountain View, CA 94040}
\email{simon@eulercircle.com}
\author{Stephen Zhou}
\email{stephenzh99@gmail.com}
\date{\today}

\begin{document}
\begin{abstract}
    We find that partisan mis\`ere quotients can have any finite cardinality other than $3$, answering a question of Allen \cite{allen2015peeking}. This contrasts with impartial mis\`ere quotients, which must have even cardinality.
\end{abstract}
\maketitle

\section{Introduction}
Games played under the mis\`ere play convention are much harder to understand than those under the normal play convention. For example in normal play, Sprague \cite{sprague} and Grundy \cite{grundy} independently discovered that any impartial game is equivalent to a nimber.
But impartial games under mis\`ere play are much more complex. Conway attempted to generalize Sprague--Grundy theory to mi\`ere play impartial games with his genus theory in Chapter 12 of \cite{conway2000numbers}, but this only gave solutions to a few rulesets. Progress on mis\`ere games stalled until Plambeck and Siegel introduced the mis\`ere quotient in \cite{plambeck2008misere}. 

Mis\`ere quotients are commutative monoids built from sets of games closed under taking options and sums. These sets are referred to as \emph{closed}. The set of games corresponding to a ruleset is usually closed. If $\mathscr{A}$ is a closed set, we say that $G \equiv_{\mathscr{A}} H$ if $o^-(G+X) =o^- (H+X)$ for all $X \in \mathscr{A}$. This is a relaxation of the usual definition of mis\`ere play equality, where $G = H$ if $o^- (G+X) = o^- (H+X)$ for all games $X $.
The mis\`ere quotient $\mathcal{Q}(\mathscr{A})$ consists of the equivalence classes of $\mathscr{A}$ under the $\equiv_{\mathscr{A}}$ relation.
It is often possible to solve rulesets whose finite mis\`ere quotients are finite, as done by Plambeck and Siegel for several mis\`ere play octal games in \cite{plambeck2008misere}. 

It was also shown in \cite{plambeck2008misere} that any mis\`ere quotient of a set of impartial games had even cardinality. This isn't true for mis\`ere quotients of partizan games: several examples of mis\`ere quotients with odd cardinality are given by Allen in \cite{allen2015peeking}. This motivates Allen to ask about the possible cardinalities of partizan mis\`ere quotients. Weimerskirch asks the same question in \cite{weimerskirchquotients}. 

In this paper, we find exactly which cardinalities are possible for finite partisan mis\`ere quotients. 
\begin{thm}
    \label{main}
    Let $n \in \mathbb{N}$, with $n \neq 3$. Then there exists a closed set $\mathscr{A}$ such that $|\mathcal{Q}(\mathscr{A})| = n$.
\end{thm}
\begin{thm}
   \label{main2}
    There is no $\mathscr{A}$ such that $|\mathcal{Q}(\mathscr{A})| = 3$.
\end{thm}
The rest of the paper will be spent proving these theorems. In Section \ref{prelims}, we recall some facts necessary for our construction of quotients of arbitrary cardinality, before giving it and proving its validity in Section \ref{construction}. We finish by proving that no quotient of cardinality three exists in Section \ref{impossibility}.
Throughout, all games will be assumed to be played under the mis\`ere play convention.
\section{Preliminaries}
\label{prelims}

First we recall the outcome classes of nim under mis\`ere play.
\begin{defn}
    Let $N = *a_1 + \cdots +  *a_n $ be a sum of nimbers. Then the \emph{Grundy value} of $N$, $\mathcal{G}(N)$ is $\mathcal{G}(N) = a_1 \oplus \cdots \oplus a_n $, where $\oplus$ is the XOR operation.
\end{defn}
Since $\oplus$ is commutative, if $N$ and $M$ are sums of nimbers, $\mathcal{G}(N+M) = \mathcal{G}(N) \oplus \mathcal{G}(M)$. It is known that if $m < \mathcal{G}(N)$, there exists an option $N'$ of $N$ such that $\mathcal{G}(N') = m$.
\begin{defn}
    Let $N = *a_1 \dots + *a_n$ be a sum of nimbers. If $\max (a_i) > 1$, we say that $N$ is \emph{firm}, otherwise we say that $N$ is \emph{fickle}.
\end{defn}

Notice that $N$ must be firm if $\mathcal{G}(N) > 1$.
\begin{thm}
    \label{nim_strategy}
    Let $N = *a_1 + \cdots +  *a_n $ be a sum of nimbers. Then \begin{equation*}
        N \in \begin{cases}
            \mathcal{N^-} & \text{ if } \mathcal{G}(N) \neq 0 \text{ and }N\text{ is firm} \\
             \mathcal{P^-} & \text{ if } \mathcal{G}(N) = 0 \text{ and }N\text{ is firm}  \\
              \mathcal{N^-} & \text{ if } \mathcal{G}(N)  = 0 \text{ and } N\text{ is fickle}\\
               \mathcal{P^-} & \text{ if } \mathcal{G}(N) \neq 0 \text{ and } N\text{ is fickle}\\
        \end{cases} 
        \end{equation*}
\end{thm}
  This is proven in  \cite[136]{conway2000numbers}.
The following Theorem is a special case of a Theorem proven in \cite{otherpaper}, which also has simpler proof.
\begin{notn}
    The notation $nG$ for integer $n$ denotes the disjunctive sum of $n$ copies of $G$.
\end{notn}
\begin{thm}
  \label{atomic_weight}
  Let $B = \{*a_1,\dots ,*a_n\mid *b_1,\dots,*b_m\}$ be a game where all moves are to nimbers. Assume that $\mex (a_i)  > \mex (b_i)>1$. Let $N$ be any sum of nimbers. Let $k>2$ be an integer. Then \[
    B + N \in \mathcal{N}^- \cup \mathcal{L}^-    
  \qquad\text{and}\qquad
      kB + N \in \mathcal{L}^-.
  \]
\end{thm}
\begin{proof}
First we will prove that $B+N \in \mathcal{N}^- \cup\mathcal{L}^-$.   We will induct on the birthday of $N$. If $N$ is $0$, then Left wins by moving to $* \in \mathcal{P}^-$. Assume the theorem has been proven for all $B+N$, where $N'$ is any subposition of $N$. Say that $\mathcal{G}(N) \leq \mex(b_i)$ and $G$ is firm. Then $\mathcal{G}(N) < \mex(a_i)$, meaning there exists an option $*a_i$ such that $\mathcal{G}(*a_i) = a_i = \mathcal{G}(N)$. By Theorem \ref{nim_strategy}, Left wins going first by moving to $*a_i + N$, since $\mathcal{G}(*a_i+N) = \mathcal{G}(a_i)\oplus\mathcal{G}(N) = \mathcal{G}(N)\oplus\mathcal{G}(N) = 0$.  Say that Say that $\mathcal{G}(N) \leq \mex(b_i)$ and $G$ is fickle. Then $\mathcal{G}(N) \in \{0,1\}$, so $\mathcal{G}(N) \oplus 1 \leq 1$ and Left wins by moving to $*(\mathcal{G}(N) \oplus 1)$, by Theorem \ref{nim_strategy}.
   
   So assume that $\mathcal{G}(N) > \mex(b_i)$. Then $N$ must be firm, and Left can win going first by moving in $N$ to a option $N'$ such that $\mathcal{G}(N') = \mex(b_i)$. By induction, $B+N'' \in \mathcal{N}^- \cup \mathcal{L}^-$ for all options $N''$, so Right cannot win by moving in $N'$. So Right must move in $B$ to some $*b_i$. But $\mathcal{G}(*b_i + N') = b_i \oplus \mex(b_i) \neq 0$, so Right loses, since $\mex(b_i) >1$, meaning $N'$ must have a nimber greater than $*$ in it.

   Now we prove that $kB + N \in \mathcal{L}^-$ for all $k>1$. Induct on $k$, and induct on the birthday of $N$ within that induction. The base case for both inductions is where $N$ is $0$ and $k = 2$. Left wins by moving to $B+*\mex(b_i)$, since Right must follow up by moving to $*b_j + *\mex(b_i)$, which is a firm sum of nimbers. The Grundy value of $*b_j + *\mex(b_i)$ is $b_j \oplus \mex(b_i) \neq 0$, so Right loses by Theorem \ref{nim_strategy}. Now assume that the Theorem has been proven for all subpositions of $kB +N$ by induction. Then Left wins $kB+N$ going first by moving to $N'$ if $N$ is not $0$, and to $(k-1)B$ is $N$ is $0$ and $k>2$. Right loses going first since he must move from $kB +N$ to $k'B + N'$, where $k '$ is positive, and induction proves that Left wins these positions going first. Thus $kB+N \in \mathcal{L}^-$ whenever $k > 1$.
\end{proof}

Using the notation of \cite{otherpaper}, we call games of form $B$ are \emph{blue mutant flowers}. Theorem \ref{atomic_weight}
 says that Left wins going first in the sum of one blue mutant flower with nimbers, and going first or second in the sum of more than one blue mutant flower with nimbers. The quotients we use will contain the closures of sets of blue mutant flowers. First, recall the mis\`ere quotients of the nimber $*2^{n-1}$, as described in \cite{plambeck2008misere}.

\begin{thm}
  \label{tame} 
    Let $\mathscr{T}_n = \cl (*2^{n-1})$. If $a$ is the equivalence class of $*$, and $b_i$ the equivalence class of $*2^i$, then the mis\`ere quotient $\mathcal{Q}(\mathscr{T}_n)$ is the monoid \begin{equation*}
        \langle a, b_1,b_2,\dots, b_{n-1} \mid a^2 = 1,b_i^3 = b_i,b_1^2 = b^2_2 = \dots = b_{n-1}^2 \rangle,
    \end{equation*}
which has cardinality $2^{n}+2$.
\end{thm}

\begin{lemma}
    \label{eval}
    Let $X = \{x_1 ,\dots,x_n\}$ be a set of nonnegative integers, and define $B = \{  0,* ,*2,\dots, *\mex (X) \mid *x_1,\dots,*x_n\}$. Let $N$ be a sum of nimbers. Then if $\mex(X) > 1$, \begin{equation}
        B +N \in \begin{cases}
            \mathcal{N}^- & \text{if } \mathcal{G}(N) \in X \\ 
            \mathcal{R}^- & \text{otherwise}\\
            
        \end{cases}
    \end{equation}
\end{lemma}
\begin{proof}
    By Theorem \ref{atomic_weight}, Left wins going first under mis\`ere play, so $B+N \in \mathcal{N}^-\cup \mathcal{R}^-$, and Right going first must move in $B$ to have a chance at winning. So to determine the outcome class out $B+N$, we only need to see if Right has a winning move in $B$.

Let $\mathcal{G}(N) = a$. If $N$ is firm, and $a \in X$, we see that $B+N \in \mathcal{N}^-$, since Right can win by moving to $*a +N$, which has Grundy value $a \oplus a = 0$.  If $N$ is firm and $a \not \in X$, $B+N \in \mathcal{L}^-$, since $\mathcal{G}(*x_i + N) = x_i \oplus a \neq 0 $ for all $i$. If $N$ is fickle, then $\mathcal{G}(N) \in X$ since $\mex(X)>1$, and Right wins by moving to $*(a \oplus 1) + N$.
\end{proof}

\section{Quotients of Sums of Blue Mutant Flowers}

\label{construction} 

We use the constructions of Theorems \ref{quotientsize} and \ref{quotientsize2} to prove Theorem \ref{main}.
Using Theorem \ref{atomic_weight}, we will prove that there exist partisan mis\`ere quotients of every finite cardinality except $3$. Then we will prove that no quotient of cardinality $3$ exists in Section \ref{impossibility}. 

\begin{defn}
    Let $N_{<2^n}$ be the set of integers strictly less than $2^n$ viewed as a vector space over $\mathbb{F}_2$, where the addition  operation is the binary XOR $\oplus$. This vector space is isomorphic to $\mathbb{F}_2^n$.
\end{defn}
\begin{thm}
    \label{quotientsize}
    Let $H_i = \{x_{1,i},\dots,x_{2^{d_i},i}\}, 1 \leq i \leq m$ be a subspace of  $N_{<2^n}$ of dimension $d_i$, with $1 \in H_i$, and let $B_i = \{0,*\mid 0,*x_{1,i},\ldots,*x_{2^{d_i},i}\}$. Let $\mathscr{A} = \cl( *2^{n-1},B_1,\ldots,B_m ) $. If all subspaces $H_i$ are distinct and contain $1$, then $\left |\mathcal{Q} (\mathscr{A})\right | = \left( 2^n +\sum_{i=1}^m 2^{n-d_i} \right) +3 $ with the following equivalence classes:

i) the $2^n+2$ equivalence classes of $\mathscr{T}_n$,

ii) The $2^{n-d_i}$ equivalence classes $B_i + X$, where $\mathcal{G}(X) \equiv a \pmod {H_i}$ for a fixed $a \in N_{<2^n}$.

iii) A single equivalence class containing all sums in $\mathscr{A}$ with at least two blue mutant flowers.
\end{thm}

\begin{thm}
\label{quotientsize2}
    Let $\mathscr{A}$ be as in Theorem \ref{quotientsize}. Then $\left|\mathcal{Q} ( \mathscr{A} \cup \{2B_1 | \varnothing\})\right|  = \left |\mathcal{Q} (\mathscr{A})\right |+1$.

    The equivalence classes are as follows: 

    i) from Theorem \ref{quotientsize}

    ii) from Theorem \ref{quotientsize}

    iii) The equivalence class iii) from Theorem \ref{quotientsize}, expanded to include all sums in $\cl(\mathscr{A} \cup \{2B_1 \mid \varnothing\})$ containing at least one copy of $\{2B_1 \mid \varnothing\}$, excluding $n\{2B_1 \mid \varnothing\}$.

    iv) One equivalence class containing all $n\{2B_1 \mid \varnothing\}$ for $n>0$.
\end{thm}

\begin{ex}
   Let $B = \{0,*,*2,*3 \mid 0,*2,*3,*4\}$, and $H = \{0,1,2,3\}$. We calculate that $N_{<8}/ H$ consists of two equivalence classes $\{0,1,2,3\}$, and $\{4,5,6,7\}$. Let the representatives of these classes be $0$ and $4$. By Theorem \ref{quotientsize}, the quotient $\mathcal{Q}(B_i,*4)$ has cardinality $8+2+3 = 13$, with representatives $0,*,\dots,*7$, $*2+*2$, $*2+*2+*$ and $B ,B+*4$. and $2B$.
\end{ex}
Both these theorems will be proven as a series of lemmas. Theorem \ref{atomic_weight} and Lemma \ref{eval}
 will be used frequently.

\begin{lemma}
    \label{4}
   Let $N,M$ be sums of nimbers. Then for any $i$ and $B_i + N\not \equiv_{\mathscr{A}} M$.
\end{lemma}

\begin{proof}
    Let $M'$ be a firm sum of nimbers with $\mathcal{G}(M') = \mathcal{G}(M)$. We calculate that $\mathcal{G}(M+M') = \mathcal{G}(M)\oplus \mathcal{G}(M') = 0$. Since $M + M'$ is a firm sum of nimbers, $M+M' \in \mathcal{P}^-$.
    By Theorem \ref{atomic_weight}, $o^- (B_i + Y +M') = \mathcal{L}^- \cup \mathcal{N}^- $. So $M'$ distinguishes $B_i + N $ and $M$.
\end{proof}

\begin{lemma}
    \label{3}
    For any $i \neq j$ and, $B_i + N \not \equiv_{\mathscr{A}} B_j +M$. 
\end{lemma}
\begin{proof}
Recall that for all $i$ and sum of nimbers $X$, $B_i + X \in \mathcal{N}^-$ if and only if $\mathcal{G}(X) \in H_i$. Let $P$ be a sum of nimbers. Then $B_i+N+N+P \in \mathcal{N}^-$ if and only if $\mathcal{G}(N+N+P) = \mathcal{G}(P) \in H_i$. Similarly, $B_j+M+N+P \in \mathcal{N}^-$ if and only if $\mathcal{G}(M+N+P) = \mathcal{G}(M) \oplus \mathcal{G}(N) \oplus \mathcal{G}(P) \in {H}_j$. 

Assume that no sum of nimbers $N+P$ distinguishes $B_i + N$ and $B_j+M$ so that $o^-(B_i+N+N+P) = o^-(B_j+M+N+P)$ for all $P$. By the previous paragraph, $\mathcal{G}(P) \in H_i$ is equivalent to $\mathcal{G}(M) \oplus \mathcal{G}(N) \oplus \mathcal{G}(P) \in H_j$. So $H_j$ is a coset of $H_i$ translated by $\mathcal{G}(M) \oplus \mathcal{G}(N)$ in $N_{M2^n}$. But this is impossible since $H_j$ is also a subspace of $H_{<2^n}$, unless $\mathcal{G}(M) \oplus \mathcal{G}(N) = 0$ and $H_i = H_j$. But all the $H_i$'s are unique by definition, so this is impossible, and there exists $P$ such that $o^-(B_i+N+N+P) \neq o^-(B_j+M+N+P)$.
\end{proof}
\begin{lemma}
    \label{5}
    If $G$ is a sum containing at least two blue mutant flowers, and $H$ is not, then $G \not \equiv_{\mathscr{A}} H$.
\end{lemma}
\begin{proof}
If $H \notin \mathcal{L}^-$, then $0$ distinguishes $G$ and $H$ since $G \in \mathcal{L}^-$ by Theorem \ref{atomic_weight}. 
If $H \in \mathcal{L}^-$, then $H$ must be the sum of a single blue flower and some nimbers, say $B_i + N$. Then by Theorem \ref{eval} there exists a sum of nimbers $M$ such that $B_i+N+M \in \mathcal{N}^-$. But $G+M \in \mathcal{L}^-$ by Theorem \ref{atomic_weight}, so $M$ distinguishes $G$ and $H$.\end{proof}

So far we have shown the following: 

\begin{lemma}
\label{same weight}

Let $G,H\in\mathscr{A}$ be such that $G \equiv_{\mathscr{A}} H$.

\begin{enumerate}
    \item If $G$ is a sum of nimbers $N$, then $H$ is another sum of nimbers $N'$
    \item If $G$ is a sum of a single blue mutant flower and some nimbers, say $B_i+N$, then $H$ is the sum of the same mutant flower with nimbers, say $B_i + N'$.
    \item If $G$ is a sum containing at least two blue mutant flowers then $H$ also contains at least two blue mutant flowers.
\end{enumerate}    
\end{lemma}

 We analyze each case separately. First, we prove all the $G$ in case 3 of Lemma \ref{same weight}  are equivalent modulo $\mathscr{A}$.
\begin{lemma}
    \label{two_blue}
    Let $G_1\in \mathscr{A}$ and $G_2 \in \mathscr{A}$ be sums of at least two blue mutant flowers and nimbers. Then $G_1 \equiv_{\mathscr{A}} G_2$, with the property that $G_1 + H \equiv_{\mathscr{A}}  G_1$ for all $H \in \mathscr{A}$.
\end{lemma}
\begin{proof}
    By Theorem \ref{atomic_weight}, $o^-(G_1+H) = o^-(G_2+H) = \mathcal{L}^-$, for all $H$, so $G_1 \equiv_{\mathscr{A}} G_2$ and $G_1 + G \equiv_{\mathscr{A}} G_1$.
\end{proof}
The second sentence of the Theorem implies that when we want to test if two elements of $\mathscr{A}$, say $G$ and $H$ are equal modulo $\mathscr{A}$, we only need to consider the outcomes of $G+X$ and $H+X$, where $X$ is a sum with at most one blue mutant flower instead of all $X \in \mathscr{A}$.

Now we handle case 1 of Lemma \ref{same weight}.
\begin{lemma}
\label{1}
    Let $N$ and $N'$ be sums of nimbers. Then $N \equiv_{\mathscr{A}} N'$ if and only if  $N \equiv_{\mathscr{T}_n} N'$.
\end{lemma}
\begin{proof}
    The only if direction is clear, since $\mathscr{T}_n \subset \mathscr{A}$, so equivalence modulo $\mathscr{A}$ implies equivalence modulo $\mathscr{T}_n$.
    
  It remains to prove the if direction. We want to prove that for any $X \in \mathscr{A}$, $o^-(N+X) = o^-(N'+X)$. By definition, $X$ is a sum of blue wildflowers and nimbers. By the remark after Theorem \ref{two_blue}, we may assume that $X$ is a sum with at most one blue mutant flower. If $X$ contains no blue wildflowers and is a sum of nimbers, then $X \in \mathscr{T}_n$ , so $o^-(N+X)=o^-(N'+X)$ since $N \equiv_{\mathscr{T}_n} N'$. We are left with the case where $X$ is the sum of a blue wildflower $B$ and a sum of nimbers, say $M$. By Theorem \ref{eval}, $o^-(N+B+M) \in \mathcal{L}^- \cup \mathcal{N}^-$ and $o^-(N' + B+M) \in \mathcal{L}^- \cup \mathcal{N}^-$, and Right must move in $B$ to have a chance at winning. Since all of Right's moves are to nimbers, and $N \equiv_{\mathscr{T_n}} N'$, we have that $o^- (N+B^R + M) = o^- (N' + B^R +M)$ for any right option $B^R$, so Right having a winning move in $N+B+M$ is equivalent to him having one in $N'+B+M$. Thus, $o^-(N+B+M) = o^-(N'+B+M)$.
\end{proof}

Finally, we finish off case 2. of Lemma \ref{same weight}.

\begin{lemma}
   \label{2}
  Let $N,M$ be sums of nimbers. Then
  for any $B_i$, $B_i + N \equiv_{\mathscr{A}} B_i+M$ if and only if  $\mathcal{G}(N) \equiv \mathcal{G}(M) \pmod {H_i}$.
\end{lemma}
\begin{proof}
   All sums of two or more blue mutant wildflowers are in $\mathcal{L}^-$. So if $P$ is a sum containing a blue mutant flower, $o^-(B_i+N +P)= o^-(B_i+M+P) = \mathcal{L}^-$. Thus we only need to consider the case where $P \in \mathscr{T}_n$, that is, when $P$ is a sum of nimbers. 
    
    By Lemma \ref{eval}, $B_i +  N+P \in\mathcal{N}^-$ if and only if $\mathcal{G}(N+P) \equiv 0 \pmod H$. So, $B_i + Y+P \in \mathcal{N}^-$ is equivalent to $\mathcal{G}(N+P) =\mathcal{G}(Y) \oplus \mathcal{G}(P) \equiv 0 \pmod {H_i}$. Symmetrically,  $o^-(B_j + M+P) = \mathcal{N}^-$ is equivalent to $\mathcal{G}(M+P) = \mathcal{G}(M) \oplus \mathcal{G}(P) \equiv 0 \pmod {H_i}$.
    
    If $\mathcal{G}(N) \equiv \mathcal{G}(M) \pmod {H_i}$, then $\mathcal{G}(N) \oplus \mathcal{G}(P) \equiv \mathcal{G}(M) \oplus \mathcal{G}(P)  \pmod {H_i}$, so $o^-(B_i + N+P) = o^-(B_j + M+P)$ whenever $P$ is a sum of nimbers. By the first paragraph, this implies that $B_i + Y \equiv_{\mathscr{A}} B_j +Z$.
\end{proof}
Now we are ready to prove Theorem \ref{quotientsize}.

\begin{proof}[Proof of Theorem \ref{quotientsize}]
We will find the equivalence class of $G$ for all $G \in \mathscr{A}$.

    If $G \in \mathscr{T}_n$ , then Lemmas \ref{same weight} shows that it can only be equivalent to another sum of nimbers. Combined with Lemma \ref{1}, this shows that the equivalence class of $G$ in $\mathscr{A}$ is the same as that of $G$ in $\mathscr{T}_n$. So there are a total of $2^n +2$ different equivalence classes for different $G \in \mathscr{P}$, and these classes are the same as in $\mathscr{T}_n$.

    If $G$ is a sum of nimbers with a single blue mutant flower $B_i$, say $B_i +N$ , then Lemmas \ref{3} and \ref{4} show that it can only be equivalent to $B_i+M$, where $M$ is another sum of nimbers, and $\mathcal{G}(N) = \mathcal{G}(M) \pmod {H_i} $. There are $2^{n-d_i}$ equivalence classes modulo $H_i$, so there are a total of $2^{n-d_i}$ equivalence classes containing a $B_i +N$ for some $N$. For a particular $N$, the equivalence class contains all $B_i + N'$, where $N' \equiv N \pmod {H_i}$

    If $G$ contains two or more blue mutant, then Lemmas \ref{same weight} and \ref{two_blue} show that the equivalence class of $G$ is all other sums that contain two or more wildflowers, adding one last equivalence class. 

    Adding up all the equivalence classes we counted, we find that there are $\left( 2^n + \sum 2^{n-d_i} \right) + 3$ equivalence classes in $\mathcal{Q}(\mathscr{A})$.
\end{proof}
After another lemma, we will be ready to prove Theorem \ref{quotientsize2}.

\begin{lemma}
    For all $n \in \mathbb{N}^+$, $n \{2B_1 \mid \varnothing\} \in \mathcal{N}^-$ and for all nonzero $G \in \mathscr{A}$,  $n \{2B_1 \mid \varnothing\}+G \in \mathcal{L}^-$.
\end{lemma}
\begin{proof}
First, we will show by induction on $n$ that $n \{2B_1 \mid \varnothing\}+G \in \mathcal{L}^-\cup\mathcal{N}^-$, where $G$ is a possibly $0$ game. 
If $n = 1$, Left wins $n \{2B_1 \mid \varnothing\}+G $ going first by moving in $\{2B_1 \mid \varnothing \}$ to $2B_1$ to get to $2B_1 + G \in \mathcal{L}^-$ by Theorem \ref{eval}. By induction, assume that $(n-1)\{2B_1\mid \varnothing \} +G \in \mathcal{L}^- \cup \mathcal{N}^-$ for all $G \in \mathscr{A}$. 
We claim that Left wins $(n-1)\{2B_1\mid \varnothing \} +G $ by moving to $2B_1$ in a $\{2B_1\mid \varnothing\}$. Then Right is to move in $(n-1)\{2B_1 \mid \varnothing\} +G+2B_1$. Since he has no options in $(n-1)\{2B_1 \mid \varnothing\}$ he must move in $G+2B_1$ to some right option $(G+2B_1)^R$. But by induction, Left wins $(n-1)\{2B_1 \mid \varnothing\} +(G+2B_1)^R$ going first. Thus $n\{2B_1 \mid \varnothing\} + G \in \mathcal{L}^- \cup \mathcal{N}^-$.

Now we can show that $n \{2B_1 \mid \varnothing\}+G \in \mathcal{L}^-$ for nonzero $G$. If $G \in \mathscr{A}$ is nonzero, then $G$ has a Right option. So Right must move to $n\{2B_1 \mid \varnothing\} + G^R$, which Left wins by the previous paragraph. Thus Right loses $n\{2B_1 \mid \varnothing\} + G$ going first, and $n\{2B_1 \mid \varnothing\} + G$ for nonzero $G$.

Since $ n \{2B_1 \mid \varnothing\}$ has no Right options, Right wins going first, so $ n \{2B_1 \mid \varnothing\} \in \mathcal{N}^-$.
\end{proof}
\begin{proof}[Proof of Theorem \ref{quotientsize2}]
The previous Lemma shows that adding $n\{2B_1 \mid \varnothing\}$ won't help distinguish two elements of $\cl ( \mathscr{A} \cup \{2B_1 | \varnothing\})$. So if $G \equiv_{\mathscr{A}} H$, we also have $G\equiv_ { \mathscr{A} \cup \{2B_1 \mid \varnothing\}} H$. So we just need to consider the equivalence classes of $n\{2B_1 \mid \varnothing\}+G$. The games $n\{2B_1 \mid \varnothing\}$ form an equivalence class of their own, but the games $n\{2B_1 \mid \varnothing\}+G$ for nonzero $G$ are equivalent to $2B_1$, since the outcome class of $n\{2B_1 \mid \varnothing\} +G$ is not changed by adding an element. Thus there is only one more equivalence class, $\{n\{2B_1\mid \varnothing \}\}$, meaning $\left|\mathcal{Q} ( \mathscr{A} \cup \{2B_1 | \varnothing\})\right|  = \left |\mathcal{Q} (\mathscr{A})\right |+1$.
\end{proof}
We can now almost finish proving Theorem \ref{main}.
\begin{thm}
    Let $N\neq  1,2,3,4,5,6,7,11$. Then there exists $\mathscr{A}$ such that $|\mathcal{Q}(\mathscr{A})| = N$.
\end{thm}
\begin{proof}

Recall that $N_{<2^n}$ is defined as the set of integers strictly less than $2^n$ viewed as a vector space over $\mathbb{F}_2$, where vector addition is $\oplus$. Notice that for $d\geq 1$, there exist $\binom{n-1}{d-1}$ subspaces of dimension $d$ that contain $1$. Now consider the binary representation of $N -3$, say \begin{equation}
    N-3 = 2^{a_0} + 2^{a_1} + \dots + 2^{a_{m}},
\end{equation}
where $a_0> a_0 > \dots > a_{m}$. Assume that $a_0 \geq 2$ and $m >0$. Set $n = a_1$ and $d_i = n - a_{i}$ for $i>0$. Since all the $a_i$'s are unique and decreasing, the $d_i$'s are unique and increasing. 

Assume that $a_m > 0$, so $d_i < n$ for all $i$.
Since $d_i \geq 1$ unconditionally and the $d_i$'s are unique, there exist distinct subspaces $H_1,H_2,\dots H_m$ of $N_{<2^n}$ such that $1 \in H_i$ and $H_i$ has dimension $d_i$ for all $i$. Let $H_i = \{x_{1,i},\dots,x_{2^{d_i},i}\}$ and define $B_i = \{0,*|0,*x_{1,i},\dots,*x_{2^{d_i},i}\}$ for $1 \leq i \leq m$. Let $\mathscr{A} = \cl( *2^{n-1},B_1,\dots,B_m ) $. Then by Theorem \ref{quotientsize}, \begin{equation}
    \left |\mathcal{Q} (\mathscr{A})\right | = \left( 2^n +\sum_{i=1}^m 2^{n-d_i} \right) +3 = \left( 2^n +\sum_{i=1}^m 2^{a_i} \right) +3 = N-3+3=N.
\end{equation}

If $a_m = 0$, so that $d_m = n$, the binary representation of $N-4$ is \begin{equation}
   N-4 = 2^{a_0} + 2^{a_1} + \dots + 2^{a_{m-1}}.
\end{equation}
We can proceed as before to get a mis\`ere quotient with cardinality $N-1$. But if we applied Theorem \ref{quotientsize2}, then we get a mis\`ere quotient with cardinality $N$ instead.

It remains to consider the case where $ m  =0$, so that $N = 2^{a_0} + 3$. Say that $a_0 > 3$, and let $n = a_0 -1$. Then $N_{<2^n}$ has at a subspace $H_1$ of dimension $1$ that contains $1$, and at least two subspaces $H_2,H_3$ of dimension $2$ that contain $1$.  Let  $B_i = \{0,*|0,*x_{1,i},\dots,*x_{2^{d_i},i}\}$ for $1 \leq i \leq 3$, and $\mathscr{A} = \cl( *2^{n-1},B_1,B_2,B_3 ) $. By Theorem \ref{quotientsize}, we see that \begin{equation}
    |\mathcal{Q}(\mathscr{A})| = \left(2^{n} + 2^{n-1} + 2^{n-1} \right)+3 = 2^{n+1}+3 = 2^{a_0}+3 =N.
\end{equation}

So unless $N \leq 2^2+2 = 6$, so that $a_0 < 2$, or $N = 2^{a_0} + 3$ for $a_0 \leq 3$, we have found $\mathscr{A}$ such that $|\mathcal{Q}(\mathscr{A})| = N$. These exceptions are exactly $N = 1,2,3,4,5,6,7,11$.
\end{proof}
Now we deal with the exceptions.
\begin{lemma}
    There exist $\mathscr{A}$ such that $|\mathcal{Q}(\mathscr{A})| = 1,2,4,5,6,7,11$.
\end{lemma}
\begin{proof}

\begin{table}
\label{sets}
    \centering
    \begin{tabular}{|c|c|}
    \hline
         $n$& $\mathscr{A}$\\
         \hline
         $1$ & $\varnothing$\\
         \hline
         $2$&$*$\\
         \hline
          $4$& $\{0,*\mid 0\} $\\
         \hline
         $5$ &$  \{ 2 \{0,*\mid 0 \} \mid \varnothing \} $\\
         \hline
         $6$&$*2$ \\
         \hline
       $7$& $1,*$  \\
        \hline
        $11$& $\{*\mid0\}, \{0\mid 0,*\}$ \\
        \hline
    \end{tabular}
    \caption{Sets $\mathscr{A}$ such that $\mathcal{Q}(\mathscr{A}) = n$.}
    \label{tab:placeholder}
\end{table}
In each row of Table \ref{sets}, we give an example of a set of games $\mathscr{A}$ such that $\mathcal{Q}(\mathscr{A}) = n$ for $n = 1,2,4,5,6,7,11$.
We will not show the computations here, but the mis\`ere quotients of Table \ref{sets} can all be computed by using a partisan version of the algorithm given by Weimerskirch in \cite{weimerskirch} for computing mis\`ere quotients of impartial games. 
\end{proof}
\section{There is No Mis\`ere Quotient of Order 3}
\label{impossibility}
In this section, we will prove Theorem \ref{main2}.

There are only three games of birthday two: $1 = \{0\mid \}, * = \{ 0\mid 0\}, \overline{1} = \{ \mid 0\}$. 
Using this, we can restrict the $\mathscr{A}$ we need to check among. 
\begin{lemma}
  \label{root}
    Let $\mathscr{A}$ be a closed set. If $|\mathcal{Q}(\mathscr{A}) | = 3$, then $\mathscr{A}$ contains exactly one of $1,*,\overline{1}$.
\end{lemma}
\begin{proof}
    Since $\mathscr{A}$ is closed, it must contain at least one game of birthday $2$. It was calculated by Allen in \cite{allenthesis} that $|\mathcal{Q}(1,\overline{1})| = \infty$ and that $|\mathcal{Q}(1,*) |=|\mathcal{Q}(\overline{1},*) |= 7$. So if $\mathscr{A}$ contains more than one of $\{1,*,\overline{1}\}$, there are more than $3$ equivalence classes. Thus $\mathscr{A}$ must contain exactly one of $\{1,*,\overline{1}\}$.
\end{proof}

So we can break the proof of Theorem \ref{main2} in to two cases, up to symmetry: one where $* \in \mathscr{A}$, and one where $1 \in \mathscr{A}$. We will prove these separately.

\begin{lemma}
    \label{no star}
    Let $\mathscr{A}$ be closed, and $* \in \mathscr{A}$. Then $|\mathcal{Q}(\mathscr{A})| \neq 3$.
\end{lemma}
\begin{proof}
   For the sake of contradiction, assume that $|\mathcal{Q}(\mathscr{A})| = 3$. Then aside from the two equivalence classes containing $0$ and $*$, there must be a third equivalence class. Let the game with smallest birthday in this equivalence class be $G$. By assumption, $G$ only has moves to games equivalent to $0$ or $*$ in $\mathscr{A}$. We will refer to the options of $G$ up to equivalence, since that doesn't change the outcome of any games. For example, we will say that Right has a move to $*$ in $G$, instead of there exists a $G^R$ such that $G^R \equiv_{\mathscr{A}} *$.
   Recall that $2n* \in \mathcal{N}^-$ and $(2n+1)*\in \mathcal{P}^-$, so $\mathcal{Q}(*)$ has two equivalence classes, $\{ 2n* : n>0\}$ and $\{(2n+1 )*: n>0\}$.
Say that $G \in \mathcal{P}^-$. Then $G+* \in \mathcal{N}^-$, since both players can win by moving to $G$, and both players can only have moves to $0 \in \mathcal{N}^-$. We also calculate that $G+G \in \mathcal{N}^-$, since both players can win by moving to $G$. Since $G$ acts just like $*$ in $\mathscr{A}$, we have that $G \equiv_{\mathscr{A}} *$.

Say that $G \in \mathcal{N}^-$. Then both sides have a move to a position equivalent to $*$. There are only two outcome classes in $\mathscr{A}$, so we need $G+* \in \mathcal{P}^- \cup\mathcal{N}^-$. 
If both sides have a move to $0$, then $G+* \in \mathcal{N}$ and thus $G+G \in \mathcal{P}^-$. It follows that $G+G+* \in \mathcal{N}^-$. Thus $G$ distinguishes $G+*$ and $G$, so there are more than $3$ equivalence classes in $\mathscr{A}$, a contradiction.
If one side has a move to $0$ but the other does not then $G+* \in \mathcal{L}^- \cup \mathcal{R}^-$, depending on which side has the move to $0$, which would imply that there are more than $3$ equivalence classes.
If not, then neither side can have a move equivalent to $0$, and $G +* = \mathcal{P}^-$. So $G+* \equiv_{\mathscr{A}} *$. We also calculate that $G+G \in \mathcal{N}^-$, since both sides can only move to $G+*$. Then we have shown $G \equiv_{\mathscr{A}} 0$.

 Say that $G \in \mathcal{L}^-\cup\mathcal{R}^-$. By symmetry, we can assume without loss of generality that $G \in\mathcal{L}^-$. Then Right must only have a move to $0$ in $G$. We see that $G+*\in\mathcal{L}^- \cup\mathcal{N}^-$, since Left wins by moving to $G$. 
 
 Say that $G$ is a Left end. Then $G+* \in \mathcal{N}^-$, since Right wins by moving to $*$. So $G+* \equiv_{\mathscr{A}} 0$. We also calculate that $G+G \in \mathcal{L}^-$, since $G+G$ is a Left end and Right loses because he can only move to $G$. So $G+G \equiv_{\mathscr{A}} G$. We also must have $*+* \equiv 0_{\mathscr{A}}$, since $*+* \in \mathcal{N}^-$. Then we conclude that $G+G+*+* \equiv_{\mathscr{A}} 0$, but also that $G+G+*+* \equiv_{\mathscr{A}} G+*+* \equiv_{\mathscr{A}} G$. Thus $G \equiv_{\mathscr{A}} 0$, a contradiction.

 Now say that $G$ isn't a Left end. Then it must be the case that $*$ is a Left option. Recall that the only Right option is $0$. 
Notice that $G+* \in \mathcal{N}^-$, since Right wins going first by moving to $*$, and Left can win going first by moving to $G$. So $G+* \equiv_{\mathscr{A}} 0$. 
Notice that $G+G \in \mathcal{P}^-\cup\mathcal{L}^-$, since Rights only option is to move to $G$, which Left wins.
If Left does have a move to $0$ in $G$, then we have that $G+G \in \mathcal{L}^-$, since Left wins by moving in $G+G$ to $0$, and we reach a contradiction just as in the case where $G\in\mathcal{L}^-$ was a Left end.
If Left does not have a move to $0$ in $G$, then we have that $G+G \in \mathcal{P}^-$, so $G+G \equiv_{\mathscr{A}} *$. 
Then it follows from the equivalences we have proven that $G+G + * \equiv_{\mathscr{A}} G+* \equiv_{\mathscr{A}} 0$, but also $G+G+* \equiv_{\mathscr{A}} G$, a contradiction.
\end{proof}

\begin{lemma}
  \label{no one}
    Let $\mathscr{A}$ be closed, and $1 \in \mathscr{A}$. Then $|\mathcal{Q}(\mathscr{A})| \neq 3$.
\end{lemma}
\begin{proof}
For the sake of contradiction, assume that $|\mathcal{Q}(\mathscr{A})| = 3$. Then aside from the two equivalence classes containing $0$ and $1$, there must be a third equivalence class. Let the game with smallest birthday in this equivalence class be $G$. By assumption, $G$ only has moves to games equivalent to $0$ or $1$ in $\mathscr{A}$. We will just refer to the options of $G$ up to equivalence, since that doesn't change the outcome of any games. For example, we will say that Right has a move to $1$ in $G$, instead of there exists a $G^R$ such that $G^R \equiv_{\mathscr{A}} 1$.

 Recall that $n1 \in \mathcal{R}^-,0\in\mathcal{N}$, so the two equivalence classes in $\mathcal{Q}(1)$ are $\{n1 : n>0\}$  and $\{0\}$.

Say that $G \in \mathcal{R}$. We will show that this is impossible.
We claim that $G+1\in \mathcal{R}$. If Right has no options in $G$, he has no options in $G+1$/ either, and he wins. If he does have a move in $G$, he wins by moving to $G^R+1$, since $G^R \equiv_{\mathscr{A}} 0,1$ and both possibilities lead to $\mathcal{R}^-$ games. Similarly, Left must have a move in $G$ since $G\in\mathcal{R}^-$, and we have $G^L \equiv_{\mathscr{A}} 0,1$, so $G^L + 1 \in \mathcal{R}^-$.
Further, we have that $G+G \in \mathcal{R}^-$. If either side has an option in $G+G$, it is to either $G$ or $G+1$, both of which Right wins. And Left must have a move in $G+G$, which proves that $G+G \in \mathcal{R}^-$. The only three equivalence classes in $\mathcal{Q}(\mathscr{A})$ have representatives $0,1,G$. So we have proven that $G$ behaves exactly like $1$, in $\mathscr{A}$ or that $G\equiv_{\mathscr{A}} 1$, a contradiction.

Say that $G \in \mathcal{N}^-$. For Left to win going first, the position must be a left end, since $0$ and $1$ are the only possible options, and Right wins going first in them. Since $G \not \equiv 0$, Right must be able to move to $1$. Then we calculate that $G+1 \in \mathcal{R}^-$, since Left's only move is to $G \in \mathcal{N}^-$, and Right can win by moving to $1+1 \in \mathcal{R}^-$. We also have that $G+G \in \mathcal{N}^-$, since $G+G$ is a Left end and Right can win by moving to $G+1 \in \mathcal{R}^-$. So we conclude that $G \equiv_{\mathscr{A}} 0$, a contradiction.

Say that $G \in \mathcal{P}^-$. Then Right only has moves to $0$ , and $G$ cannot be a Left end.
We calculate that $G+1\in \mathcal{N}^-$, since Right wins by moving to $1 \in \mathcal{R}^-$, and Left wins moving to $G \in \mathcal{P}^-$.
Since the only equivalence class in $\mathcal{N}^-$ has representative $0$, $G+1 \in \mathcal{N}^-$ is equivalent to $G+1 \equiv_{\mathscr{A}} 0$. 
Now consider $G+G$. Right is able to win going first by moving to $G \in \mathcal{P}^-$.
If $0$ is a Left option of $G$, $G+G \in \mathcal{N}$, since Left can also win by going to $G$. As before, this implies that $G+G \equiv_{\mathscr{A}} 0$. Then $G+G+1 \equiv_{\mathscr{A}} 1$, but $G+1 \equiv_{\mathscr{A}} 0$ implies that $G+G+1 \equiv_{\mathscr{A}} G$ as well. We have reached $G \equiv_{\mathscr{A}} 1$, a contradiction.
If $0$ is not a left option of $G$, Left's only move in $G+G$ is to $G+1 \in \mathcal{N}^-$, so $G+G \in \mathcal{R}^-$ and $G + G\equiv_{\mathscr{A}} 1$. This implies $G+G+1 \equiv_{\mathscr{A}} 1$, but $G+1 \equiv_{\mathscr{A}} 0$ means that $G+G+1 \equiv_{\mathscr{A}} G$, a contradiction. 

Say that $G \in \mathcal{L}^-$. Then $G$ must be a left end. Right cannot win $G $ moving first, so he must only have a move to $0$. Then $G +1 \in \mathcal{N}^-$, since Left wins by moving to $G \in \mathcal{L}^-$, and Right by moving to $1 \in \mathcal{R}^-$. So $G+1 \equiv_{\mathscr{A}} 0$. So $G+1 + 1 +1\equiv_{\mathscr{A}} 1 +1 \equiv_{\mathscr{A}} 1$. We can also calculate that $G+G+1+1+1 \equiv_{\mathscr{A}} G$, but also that $G + (G+1+1+1) \equiv_{\mathscr{A}} G+(1) \equiv_{\mathscr{A}} 0$, a contradiction.
\end{proof}
Up to symmetry, combining Lemmas \ref{root},\ref{no star}, and \ref{no one} proves Theorem \ref{main2}.

\printbibliography
\end{document}